\documentclass[11pt]{article}
\usepackage{amsmath,amsfonts,amsthm,amssymb}

\usepackage[margin=4.25cm,top=3.25cm,bottom=4.75cm]{geometry}

\usepackage{color} 
\newcommand{\blue}[1]{{#1}}

\usepackage{adjustbox}
\newcommand{\specialcell}[2][c]{\begin{tabular}[#1]{@{}c@{}}#2\end{tabular}}

\usepackage{tikz-cd}

\usepackage{tikz}  %for more complicated reaction networks
\usetikzlibrary{shapes,automata,positioning,arrows,fit,snakes}
\tikzset{every node/.style={auto}}
\tikzset{every state/.style={rectangle, minimum size=0pt, draw=none, font=\normalsize}}

\usepackage{float}

\newcommand{\R}{\mathbb{R}}
\newcommand{\Z}{\mathbb{Z}}
\newcommand{\RR}{\mathcal{R}}
\newcommand{\T}{\mathcal{T}}

\providecommand{\abs}[1]{\lvert#1\rvert}
\newcommand{\dd}[2]{\frac{\text{d} #1}{\text{d} #2}}
\DeclareMathOperator{\supp}{supp}

\newcommand{\lras}{\rightleftarrows}
\newcommand{\lrl}{{\; -\hspace{-1.25ex}- \;}}

\newtheorem{thm}{Theorem}
\newtheorem{pro}[thm]{Proposition}

\newtheorem{rem}[thm]{Remark}
\newtheorem*{remark*}{Remark}
\newtheorem{exa}[thm]{Example}

\makeatletter
\def\blfootnote{\xdef\@thefnmark{}\@footnotetext}
\makeatother

\usepackage{hyperref}

\begin{document}

\title{\Large Detailed balance = complex balance + cycle balance. \\
A graph-theoretic proof for reaction networks \\ 
%and its relevance for 
and Markov chains}

\author{Stefan M\"uller and Badal Joshi}
\blfootnote{
\scriptsize

\vspace{-2ex}
%\noindent
%$^*$Corresponding author

\smallskip
\noindent
{\bf S.~M\"uller} (\href{mailto:st.mueller@univie.ac.at}{st.mueller@univie.ac.at}),
Faculty of Mathematics, University of Vienna
%Oskar-Morgenstern-Platz 1, 1090 Wien, Austria

\smallskip
\noindent
{\bf B.~Joshi}, \href{mailto:bjoshi@csusm.edu}{bjoshi@csusm.edu}),
Department of Mathematics, California State University San Marcos
}

\maketitle

\begin{abstract}
We further clarify the relation between detailed-balanced and com\-plex-balanced equilibria
of reversible chemical reaction networks.
Our results hold for arbitrary kinetics and also for boundary equilibria.

%Most importantly, 
Detailed balance, complex balance, ``formal balance'', and the new notion of ``cycle balance''
are all defined in terms of the underlying graph.
%correspond to vertex balance, edge balance, and the absence of directed cycles in an induced (mixed) graph,
%determined by the network and the equilibrium.
This fact allows elementary graph-theoretic (non-algebraic) proofs of 
a previous result (detailed balance = complex balance + formal balance), 
our main result (detailed balance = complex balance + cycle balance),
and a corresponding result in the setting of continuous-time Markov chains. \\[2ex]
\end{abstract}

%----------------------------------
\section{Introduction}

Detailed balance and complex balance are important concepts
in chemical reaction network theory (CRNT).
Both principles have been proposed already in the 1870s and 1880s by Ludwig Boltzmann 
in the kinetic theory of gases (where complex balance is called semi-detailed balance)~\cite{Boltzmann1872,Boltzmann1887}.
Around 1900, Rudolf Wegscheider introduced the principle of detailed balance in the field of chemical kinetics
(and obtained the necessary conditions on the rate constants named after him)~\cite{Wegscheider1901}.
Only in the 1970s, Horn and Jackson developed the concept of complex balance
\blue{(as a generalization of detailed balance)}
in modern CRNT~\cite{HornJackson1972}.

\blue{Complex-balanced (CB) mass-action systems display remarkably robust dynamics.
If one positive equilibrium is CB, then so is every other equilibrium, % of the system, 
which justifies calling the entire system CB.
Moreover, there is exactly one positive equilibrium in every stoichiometric class (invariant set),
and this equilibrium is asymptotically stable (implied by a strict Lyapunov function)~\cite{HornJackson1972}. 
In various important cases,
it has been shown that positive CB equilibria are globally stable~\cite{anderson2011proof,craciun2013persistence}, 
a property that is conjectured to hold for all CB systems~\cite{horn1974dynamics,craciun2015toric}. 
Finally, mass-action systems that are not CB may be dynamically equivalent to CB systems
and have all their strong properties~\cite{craciun2020efficient}.}

%Complex-balanced (CB) systems have remarkably robust dynamics. If one positive equilibrium is CB, then so is every other equilibrium of the system, which justifies calling the entire system CB. Within an invariant set, where all points satisfy the same conservation relations, there is one and only one equilibrium, which, furthermore, is positive. A strict Lyapunov function is known to exist, implying that this equilibrium is asymptotically stable \cite{HornJackson1972}. It has been shown in various important cases that this CB equilibrium is globally stable \cite{anderson2011proof,craciun2013persistence}, a property that is conjectured to hold for all CB systems \cite{horn1974dynamics,craciun2015toric}. 
%There are classes of systems which are not CB, but dynamically equivalent to them, which means that they have all the strong properties of CB systems \cite{craciun2020disguised}.
%When modeled as a continuous-time Markov chain, the stationary distribution of any CB system is a product-form Poisson distribution \cite{anderson2010product}, and only CB systems have this distribution \cite{cappelletti2016product}. If the CB system is also detailed balanced (DB), then its stationary distribution is detailed balanced as a Markov chain \cite{Joshi2015}. 

For mass-action kinetics,
complex balance has been characterized by Horn \cite{Horn1972},
and explicit conditions on the ``tree constants'' of the underlying graph have been provided by Craciun et al~\cite{CraciunDickensteinShiuSturmfels2009};
see also~\cite{Johnston2014,MuellerRegensburger2014}.
Detailed balance has been characterized by Feinberg and Schuster and Schuster~\cite{Feinberg1989,SchusterSchuster1989}.
Feinberg obtains two classes of conditions on the equilibrium constants:
$\gamma = r - m + \ell$ ``circuit conditions'' and $\delta = m - \ell - s$ ``spanning forest conditions''.
Thereby, $\delta$~is the deficiency of the network~\cite{Feinberg1972}, 
and $\gamma$~is the cycle rank (cyclomatic number) of the underlying (undirected) graph~\cite{Berge1962}.
That is, $r$ is the number of reversible reactions (pairs of edges), $m$ is the number of complexes (vertices),
$\ell$ is the number of linkage classes (connected components), and $s$ is the rank of the stoichiometric matrix.
Schuster and Schuster consider ``generalized mass-action kinetics'' 
in the sense that the net reaction rate contains a mass-action factor (as for enzyme kinetics). 
They provide ``generalized Wegscheider's conditions'' on the equilibrium constants;
in fact, they obtain $r-s \; (=\gamma+\delta)$ independent conditions. 
Finally, Dickenstein and Perez-Millan have shown that,
given the circuit conditions (``formal balance''),
the conditions on the tree constants (complex balance) agree with the spanning forest conditions on the equilibrium constants (detailed balance).
That is, detailed balance is equivalent to complex balance plus formal balance,
and the result can be extended from mass-action to ``general kinetics''~\cite{DickensteinPerezMillan2011}.
For mass-action, an alternative proof has been given in~\cite{vanderSchaft2015}.
\blue{For stochastic mass-action,
the stationary distribution of the resulting continuous-time Markov chain is a product-form Poisson distribution 
if and only if the underlying deterministic system is CB~\cite{anderson2010product,cappelletti2016product}. 
If a CB system is also detailed-balanced, then the stationary solution is detailed-balanced (reversible)~\cite{Joshi2015}. 
For other aspects of detailed and complex balance, see e.g.~\cite{MuellerHofbauer2015,FeliuCappellettiWiuf2018}.}

%In the first part of the paper, 
In this work,
we provide new conditions on a complex-balanced equilibrium of a reversible chemical reaction network 
to be detailed-balanced. 
As just stated,
a characterization has already been obtained in~\cite{DickensteinPerezMillan2011}.
%where it was shown that detailed balance is equivalent to complex balance plus formal balance. 
On the one hand, we give an elementary graph-theoretic (non-algebraic) proof of the previous result
(without using the conditions on the tree/equilibrium constants for complex/detailed balance).
On the other hand, we show that complex balance plus a condition significantly weaker than formal balance, 
namely the absence of directed cycles in an induced (mixed) graph, is equivalent to detailed balance.
%Our conclusions hold on a graphical level (for any type of reversible graph) and for general kinetics (mass-action or not). In particular, they can be applied to Markov chains.
The result immediately holds for arbitrary kinetics %(mass-action, general kinetics, \ldots)
and also for boundary equilibria.
Since our proof is based on the induced graph,
it can be applied in other settings with an underlying graph structure.  
We illustrate this via continuous-time Markov chains.
%Our main tool is the notion of an induced graph which may be defined for any mathematical model which has an underlying graphical structure and a real-valued function assigned to each node and each edge. Examples of such settings include chemical reaction networks, 

%In the second part of the paper, we provide a quantification of the amount by which a complex-balanced network deviates from being a detailed-balanced network. In this part, we only require that the network be weakly reversible, since the network is not necessarily detailed-balanced. The results in this section may also be applied to Markov chains and other mathematical settings where there is an underlying network structure, however we omit the details since it is straightforward to translate them to different settings. 

The work is organized as follows. 
First, we present the elementary argument (balance in mixed graphs) that is common to all types of networks.
Then, we apply it to different types of networks (balance in reaction networks and balance in Markov chains).

%----------------------------------
\section{Balance in mixed graphs} \label{sec:mix}

The object of study in this section is a simple mixed graph.
\blue{
Recall that a {\em mixed} graph contains undirected and directed edges, in general, 
and that a {\em simple} mixed graph does not contain multiple edges (connecting two vertices) or loops (connecting a vertex to itself).}

%\red{Is it necessary to require ``simple'' graph? Should work with small modifications even when there are loops and multiple edges. }

Let $G=(V,U,D)$ be a simple mixed graph
(with vertices~$V$, undirected edges~$U$, and directed edges~$D$).
\blue{Explicitly, if two vertices $v, v'\in V$ are connected by an edge,
then $v \neq v'$
and exactly one of the following holds: $(v \lrl v') \in U$, $(v \to v') \in D$, or $(v \leftarrow v') \in D$.}

A {\em path} is a \blue{(finite or infinite)} sequence of edges which connect \blue{a sequence of} distinct vertices.
For finite paths, the first and last vertex may be identical, in which case the path is a {\em cycle}.
A path is called {\em directed} if it contains only directed edges
and all edges have the same direction \blue{(along the path)}.
\blue{In other words, 
a path connecting the vertices $v,v',v'',\ldots$ is directed if $v \to v' \to v'' \to \ldots$ or $v \leftarrow v' \leftarrow v'' \leftarrow \ldots $.}
A path is called {\em weakly directed} if it contains a directed edge
and all directed edges have the same direction.

An edge is called {\em balanced} if it is undirected.
A vertex is called {\em balanced}
if the set of incident edges contains either only undirected edges or a pair of oppositely directed edges \blue{(with respect to the vertex)}.
\blue{In other words, 
a vertex $v$ is balanced if the existence of $v'$ with $v' \to v$ implies the existence of $v''$ with $v \to v''$ and vice versa. Note that $v' \neq v''$ by the simplicity of the graph.}
%Analogously, a cycle $C \subseteq E$ is called {\em balanced}
%if it contains either only undirected edges or a pair of oppositely directed edges.
% A cycle is called {\em quasi balanced} if it is not directed.

$G$ is called {\em edge-balanced/vertex-balanced} if every edge/vertex is balanced. 
% $G$ is called cycle-balanced if every cycle is balanced. 
% $G$ is {\em quasi cycle-balanced} if no cycle is directed.
% so in other words, every cycle contains an undirected edge or two oppositely directed edges. 

\subsection{Finite graphs}

An edge-balanced graph has only undirected edges and therefore is vertex-balanced and contains no directed cycle. In the following, we show the converse. 

\begin{pro} \label{pro:main}
Let $G=(V,U,D)$ be a finite, simple mixed graph.
If $G$ is vertex-balanced, but not edge-balanced,
then it contains a directed cycle.  
\end{pro}
\begin{proof}
Assume that $G$ is vertex-balanced and that there exists a directed edge $v \to v'$.
By vertex-balance for $v'$, there exists a corresponding directed edge $v' \to v''$.
Repeating this argument, we construct a directed path $v \to v' \to v'' \to \ldots$ which,
by the finiteness of the graph,
eventually yields a directed cycle.
\qed
\end{proof}
The main result used in the following section is the contrapositive of Proposition~\ref{pro:main}, which we state as a theorem. 
\begin{thm} \label{thm:main}
Let $G=(V,U,D)$ be a finite, simple mixed graph.
If $G$ is vertex-balanced and contains no directed cycle,
then it is edge-balanced.
\end{thm}

\subsection{Infinite graphs}

\blue{A directed path is called {\em bi-infinite}
if it connects a bi-infinite sequence of vertices.}
Bi-infinite directed paths can be viewed as a ``\blue{directed} cycles of infinite length''.
%\blue{By a {\em bi-infinite directed path}, we mean a directed path which has infinitely many edges, and no initial or terminal vertex.}

\begin{pro}
Let $G=(V,U,D)$ be a simple mixed graph.
If $G$ is vertex-balanced, but not edge-balanced,
then it contains a directed cycle or a bi-infinite directed path.  
\end{pro}
\begin{proof}
Analogous to the proof of Proposition~\ref{pro:main}.
\end{proof}

Again, as a main result, we state its contrapositive.

\begin{thm} \label{thm:main_inf}
Let $G=(V,U,D)$ be a simple mixed graph.
If $G$ is vertex-balanced and contains no directed cycle or bi-infinite directed path,
then it is edge-balanced.
\end{thm}

As a consequence, if $G$ is vertex-balanced and contains no directed cycle, then it cannot have a finite number of directed edges.

%----------------------------------
\section{Balance in reaction networks}

\blue{In the following,
we denote the positive real numbers by $\R_{>}$ and the nonnegative real numbers by $\R_{\ge}$. 
For a vector $x \in \R^n$, we denote its support by $\supp(x) = \{ i \mid x_i \neq 0 \}$.
For $x,y \in \R^n_\ge$, we define $x^y = \prod_{i=1}^n (x_i)^{y_i} \in \R_\ge$.

A {\em chemical reaction network} $(G,y)$ is given by} a finite, simple directed graph $G=(V,\RR)$
\blue{and a map $y \colon V \to \R^n_\ge$.}
%\begin{remnn}
%The graph $G=(V,\RR)$ is called the {\em abstract reaction graph}.
To each vertex $i \in V$, a vector {\em (complex)} $\blue{y(i)} \in \R^n_\ge$ is assigned.
%The resulting graph $G'=(y(V),\RR')$ is called the {\em complex-reaction graph}.
%Thereby,
%$\left( y(i) \to y(j) \right) \in \RR' \subseteq y(V) \times y(V)$
%if and only if $(i \to j) \in \RR \subseteq V \times V$.
\blue{Complexes represent formal sums of $n$ chemical species
which are the left- and right-hand sides of chemical reactions.
%Every chemical reaction $y(i) \to y(j)$ (in the complex-reaction graph) 
%corresponds to the edge $i \to j$ (in the abstract reaction graph).

As an example, consider the ``network'' consisting of the single reaction $\sf A + B \to C$,
involving the three species $\sf A, B, C$.
The underlying graph has two vertices, say 1 and 2, and one edge, $1 \to 2$,
that is, $G=(\{1,2\},\{1\to2\})$.
The left-hand side of the reaction is a formal sum of species $\sf A$ and $\sf B$,
and the right-hand side equals species $\sf C$,
that is,
they are represented by the complexes $y(1)=(1,1,0)^T$ and $y(2)=(0,0,1)^T$, respectively.
%\end{remnn}

A {\em kinetic system} $(G,y,r)$ is given by a chemical reaction network $(G,y)$, where $G=(V,\RR)$,
and a map $r \colon \RR \to (\R^n_\ge \to \R_\ge)$.}
To each edge $(i \to j) \in \RR$, a rate function {\em (kinetics)} $r_{i \to j} \colon \R^n_\ge \to \R_\ge$ is assigned.

The resulting dynamical system for the concentrations $x \in \R^n_{\ge}$ (of $n$ chemical species) is defined as
\begin{equation} \label{dynsys}
\dd{x}{t} = \sum_{(i \to j) \in \RR} \big( y(j)-y(i) \big) \, r_{i \to j} (x) .
\end{equation}

\begin{remark*}
For ``general kinetics'', it is often assumed that $r_{i \to j}(x) > 0$ if and only if $\supp(y(i))\subseteq\supp(x)$.
Then, $x \in \R^n_>$ implies $r(x)\in\R^\RR_>$.
For mass-action kinetics,
the complexes determine not only the reaction vector $y(j)-y(i)$,
but also the reaction rate
\[
r_{i \to j}(x) = k_{i \to j} \, x^{y(i)} \quad \text{for } (i \to j) \in \RR .
\]
\end{remark*}

In the following, we consider {\em reversible} reaction networks,
where the underlying graph $G$ is symmetric,
that is, $(i \to j) \in \RR$ if and only if $(j \to i) \in \RR$.
For simplicity, we often write $ij$ for $i \to j \in \RR$.

\subsection{Detailed and complex balance}

An equilibrium $x \in \R^n_\ge$ of the dynamical system~\eqref{dynsys} is called {\em detailed-balanced (DB)}
if, for every $ij \in \RR$,
\[
r_{ij}(x) = r_{ji}(x) .
\]
That is, for every (reversible) reaction, the forward and backward rates are equal.

An equilibrium $x \in \R^n_\ge$ of the dynamical system~\eqref{dynsys} is called {\em complex-balanced (CB)}
if, for every $i \in V$,
\[
\sum_{ij \in \RR} r_{ij}(x) = \sum_{ji \in \RR} r_{ji}(x) .
\]
That is, for every complex, the sums of incoming and outgoing rates are equal.

Obviously, we have the implication
\begin{equation}
x \text{ is DB} 
\quad\implies\quad
x \text{ is CB}.
\end{equation}

\subsection{Formal balance and other variants of cycle balance}

A {\em directed} cycle $C \subseteq \RR$ is a %cyclic 
sequence of edges which connect \blue{a cyclic sequence of} distinct vertices
(except that the first and last vertex are identical)
and which have the same direction \blue{(along the cycle)}. 
Reversible reactions are directed two-cycles (connecting two vertices),
and all cycle conditions below hold trivially for directed two-cycles.

A state $x \in \R^n_\ge$ (not necessarily an equilibrium) of the dynamical system~\eqref{dynsys} is called {\em formally balanced (FB)}
if, for every directed cycle $C \subseteq \RR$,
\begin{equation*} \label{eq:fb}
\prod_{ij \in C} r_{ij}(x) = \prod_{ij \in C} r_{ji}(x) ,
\end{equation*}
cf.\ \cite{DickensteinPerezMillan2011}.
Alternatively, such a state could be called algebraically cycle-balanced;
see also the discussion in the setting of Markov chains \cite{CappellettiJoshi2018}.

\begin{remark*} 
Under quite weak assumptions on the kinetics, 
formal balance is independent of the state:
With every vertex $i \in V$ associate a function $f_i(x)$, with every edge $ij \in \RR$ a function $k_{ij} \, g_{ij}(x)$,
and assume that the reaction rates can be written as $r_{ij}(x) = k_{ij} \, g_{ij}(x) \, f_i(x)$.
Now, let $r(x)\in\R^\RR_>$.
If $g_{ij}(x) = g_{ji}(x)$ for every $ij \in \RR$ or, even more generally, if
$
\prod_{ij \in C} g_{ij}(x) = \prod_{ij \in C} g_{ji}(x)
$
for every directed cycle $C \subseteq \RR$,
then formal balance amounts to 
\[
\prod_{ij \in C} k_{ij} = \prod_{ij \in C} k_{ji} 
\]
for every directed cycle $C \subseteq \RR$.
For mass action,
$f_i(x) = x^{y(i)}$ and $g_{ij}(x) = 1$.
%and $r_{ij}(x) = k_{ij} \, x^{y(i)}$.
For ``generalized mass action'' in the sense of reversible enzyme kinetics~\cite{SchusterSchuster1989},
$f_i(x) = x^{y(i)}$, $g_{ij}(x) = g_{ji}(x)$,
and hence $r_{ij}(x) - r_{ji}(x) = g_{ij}(x)(k_{ij} \, x^{y(i)}-k_{ji} \, x^{y(j)})$.
In both cases, formal balance
only depends on the rate constants (for $x\in\R^n_>$).
\end{remark*}

Formal balance is defined by equations for directed cycles.
We introduce two other variants of cycle balance which are defined by inequalities and which are weaker than formal balance.

A state $x \in \R^n_\ge$ of the dynamical system~\eqref{dynsys} is called {\em strongly cycle-balanced (sCycB)}
if, for every directed cycle $C \subseteq \RR$,
either $r_{ij}(x) = r_{ji}(x)$ for all $ij \in C$
or there exist $ij \in C$ and $i'j' \in C$ with
\[
r_{ij}(x) < r_{ji}(x) \quad \text{and} \quad r_{i'j'}(x) > r_{j'i'}(x) .
\]
%$r_{ij}(x) < r_{ji}(x)$ and $r_{i'j'}(x) > r_{j'i'}(x)$.

A state $x \in \R^n_\ge$ of the dynamical system~\eqref{dynsys} is called {\em cycle-balanced (CycB)}
if, for every directed cycle $C \subseteq \RR$,
there exist (not necessarily distinct) $ij \in C$ and $i'j' \in C$ with
\[
r_{ij}(x) \le r_{ji}(x) \quad \text{and} \quad r_{i'j'}(x) \ge r_{j'i'}(x) .
\]
%$r_{ij}(x) \le r_{ji}(x)$ and $r_{i'j'}(x) \ge r_{j'i'}(x)$.

For arbitrary kinetics, we have the implications
\begin{equation} \label{eq:impl1}
\begin{array}{ccccc}
x \text{ is DB} & \implies & x \text{ is FB} \\
& \rotatebox[origin=c]{-45}{$\implies$} & \phantom{xx} \rotatebox[origin=c]{-90}{$\implies$} \scriptstyle \, (\ast)
& \rotatebox[origin=c]{-45}{$\implies$} \\
&& x \text{ is sCycB} & \implies & x \text{ is CycB}.
\end{array}
\end{equation}

Thereby, implication $(\ast)$ holds for $r(x)\in\R^\RR_>$.
All other implications hold for $r(x)\in\R^\RR_\ge$ (possibly involving zero reaction rates),
that is, for all $x \in \R^n_\ge$.

The implication ``$x$ is FB $\Rightarrow$ $x$ is CycB'' is obvious if $r(x)\in\R^\RR_>$.
Otherwise, consider a directed cycle $C \subseteq \RR$
and $r_{ij}(x)=0$ for some $ij \in C$.
Now, ``$x$ is FB'' implies $r_{j'i'}(x)=0$ for some $i'j' \in C$,
and hence $0=r_{ij}(x) \le r_{ji}(x)$ and $r_{i'j'}(x) \ge r_{j'i'}(x)=0$, that is, ``$x$ is CycB''.
All other implications are obvious.

\begin{remark*}
For ``general kinetics'', where $r_{ij}(x) > 0$ if and only if $\supp(y(i))\subseteq\supp(x)$,
in particular, for mass-action kinetics,
implication $(\ast)$ in \eqref{eq:impl1} holds for $x \in \R^n_\ge$.

\blue{
To see this, first note that the sign of $r_{ij}(x)$ is determined by $\supp(y(i))$
and hence by vertex $i$ only.
If $\supp(y(i))\subseteq\supp(x)$, we write $r_{i*}(x)>0$ (meaning that $r_{ij}(x)>0$ for all $j$ with $ij \in \RR$); 
otherwise, we write $r_{i*}(x)=0$.

Obviously, implication $(\ast)$ in \eqref{eq:impl1} holds for $x \in \R^n_>$.
It remains to consider a directed cycle $C \subseteq \RR$
with $r_{ij}(x)=0$ for some $ij \in C$.
If $r_{i'j'}(x)=0$ for all $i'j' \in C$ (and hence $r_{i'*}(x)=0$ for all vertices $i'$ in $C$), 
then also $r_{j'i'}(x)=0$ for all $i'j' \in C$,
and both ``$x$ is FB'' and ``$x$ is sCycB''.
Otherwise, $r_{i'j'}(x)>0$ for some $i'j' \in C$.
In particular, 
there is a path $j_1 \to i_1 \to \ldots \to i_\ell \to j_\ell \subseteq C$
involving the complexes $i_l$ with $r_{i_l *} (x)=0$ for $l=1, \ldots, \ell$
and the (not necessarily distinct) complexes $j_1$ and $j_\ell$ with $r_{j_1 *}(x)>0$ and $r_{j_\ell *}(x)>0$.
Hence, $0=r_{i_1 j_1}(x)<r_{j_1 i_1}(x)$ and $0 = r_{i_\ell j_\ell}(x) < r_{j_\ell i_\ell}(x)$,
and both ``$x$ is FB'' and ``$x$ is sCycB''.
}
\end{remark*}

\blue{As stated above, the two new variants of cycle balance are weaker than formal balance, in general.
They allow elementary graph-theoretic proofs of a previous result
and of a new result which holds for arbitrary kinetics and boundary equilibria; see Theorem~\ref{thm:rrn2} below.

Algorithmically, all variants of cycle balance (including formal balance) are equally costly:
the most expensive step is the identification of all cycles in the underlying graph.
For mass action (or ``generalized mass action'' in the sense of reversible enzyme kinetics~\cite{SchusterSchuster1989}) and positive states,
formal balance only depends on the rate constants. %(and not on the state).
In this case, also (strong) cycle balance does not depend on the state,
which may allow to determine the directions of the net reactions;
see Example~\ref{exaRN} below.}

\subsection{The induced graph}

Given a reversible reaction network, defined by a finite, simple directed graph $G=(V,\RR)$,
and a state $x \in \R^n_\ge$, % \blue{(not necessarily an equilibrium)},
the {\em induced graph} $G_x=(V,U,D)$ is a finite, simple mixed graph
\blue{(with vertices~$V$, undirected edges~$U$, and directed edges~$D$)}
defined as
\begin{align*}
(i \lrl j) \in U & \quad \text{if } (i \to j) \in \RR \text{ and } r_{ij}(x) = r_{ji}(x) ,\\
(i \to j) \in D & \quad \text{if } (i \to j) \in \RR \text{ and } r_{ij}(x) > r_{ji}(x) .
\end{align*} 
The induced graph contains at most one edge between any two vertices,
and hence cycles in $G_x$ connect three or more vertices.

%%%%%%%%%%

Let $x \in \R^n_\ge$ be a state of the dynamical system~\eqref{dynsys}
and $G_x$ be the induced graph.
From the definitions in Section \ref{sec:mix},
we have the implications
%\begin{equation} \label{eq:impl2}
%\begin{aligned}
%& x \text{ is DB} && \iff && G_x \text{ is edge-balanced,} \\
%& x \text{ is CB} && \,\implies && G_x \text{ is vertex-balanced,} \\
%& x \text{ is sCycB} && \iff && G_x \text{ does not contain a weakly directed cycle,} \\
%& x \text{ is CycB} && \iff && G_x \text{ does not contain a directed cycle.}
%\end{aligned}
%\end{equation}
\begin{equation} \label{eq:impl2}
\begin{array}{lrl}
x \text{ is DB} & \iff & G_x \text{ is edge-balanced,} \\[1ex]
x \text{ is CB} & \implies & G_x \text{ is vertex-balanced,} \\[1ex]
x \text{ is sCycB} & \iff & G_x \text{ does not contain a weakly directed cycle,} \\[1ex]
x \text{ is CycB} & \iff & G_x \text{ does not contain a directed cycle.}
\end{array}
\end{equation}
\blue{Note that the second implication is not an equivalence; see~Remark~\ref{rem:equivalence} below.}

\subsection{Main results}

As stated in the introduction, it was shown in \cite{DickensteinPerezMillan2011} that detailed balance is equivalent to complex balance plus formal balance. We prove that detailed balance is equivalent to complex balance plus cycle balance. 

\begin{pro} \label{pro:rrn1}
Let $x \in \R^n_\ge$ be an equilibrium of the dynamical system~\eqref{dynsys}.
If $x$ is CB and CycB, then it is DB. 
\end{pro}
\begin{proof}
By the implications \eqref{eq:impl2} and Theorem~\ref{thm:main}:
\begin{center}
\begin{tabular}{ccc}
$x$ is CB and CycB & $\implies$ & \parbox[c]{6.2cm}{\centering $G_x$ is vertex-balanced \\ and does not contain a directed cycle} \\[-1.25ex]
& & \rotatebox{-90}{$\implies$} \\
$x$ is DB & $\Longleftarrow$ & $G_x$ is edge-balanced
\end{tabular}
\end{center}
\qed
%\vspace{-4.75ex} % big gap
\end{proof}

The above result is new and stronger than the existing result:
first, it holds for $x \in \R^n_\ge$;
and second, formal balance is stronger than cycle balance, see~\eqref{eq:impl1}.
However, the main advantage from our perspective is its elementary proof, 
which is entirely graph-theoretic and does not involve any algebraic argument; 
in particular, it does not assume mass-action kinetics. 

To summarize, given complex balance, detailed balance is equivalent to all variants of cycle balance.
The result holds for $x \in \R^n_\ge$,
that is, also for boundary equilibria.

\begin{thm} \label{thm:rrn2}
Let $x \in \R^n_\ge$ be a complex-balanced (CB) equilibrium of the dynamical system~\eqref{dynsys}.
The following statements are equivalent:
\begin{itemize}
\item $x$ is detailed-balanced (DB).
\item $x$ is formally balanced (FB). 
\item $x$ is strongly cycle-balanced (sCycB).
\item $x$ is cycle-balanced (CycB).
\end{itemize}
\end{thm}
\begin{proof}
By the implications \eqref{eq:impl1} and Proposition~\ref{pro:rrn1}.
\qed
\end{proof}

\begin{rem} \label{rem:equivalence}
Only the second implication in \eqref{eq:impl2} is not an equivalence.
In order to obtain an equivalence,
we define $x \in \R^n_\ge$ to be {\em weakly complex-balanced (wCB)}
if $G_x$ is vertex-balanced.
Then, 
``$x$ is wCB $\Leftrightarrow$ $G_x$ is vertex-balanced'',
and Proposition~\ref{pro:rrn1} and Theorem~\ref{thm:rrn2} also hold
if the CB equilibrium is replaced by a wCB equilibrium.
\end{rem}

\begin{exa} \label{exaRN}
Consider the reversible cyclic network $G^\triangleright \colon \sf A \lras B \lras C \lras A$
and assume that the (isolated) network follows the laws of thermodynamics.
Adding the exchange reactions $\sf A \lras 0 \lras C$ (putting $G^\triangleright$ in a flow reactor)
yields the network~$G$,
which contains two independent cycles; see the left diagram.
Both networks have deficiency zero: $\delta^\triangleright = 3-1-2=0$ and $\delta=4-1-3=0$, respectively.
For simplicity, assume mass-action kinetics.
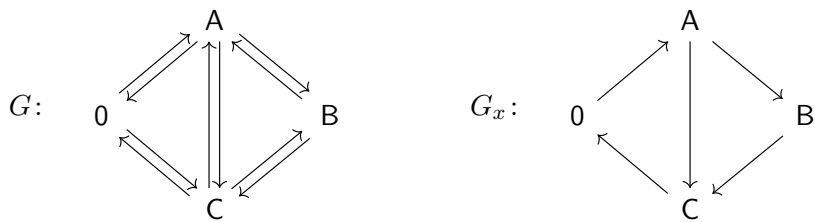
\[
G \colon \quad
\begin{tikzcd}
& \sf A \arrow[ld,xshift=0.4ex] \arrow[dd,xshift=0.425ex] \arrow[rd,xshift=0.2ex,yshift=0.6ex] \\
\sf 0 \arrow[ru,xshift=-0.2ex,yshift=0.6ex] \arrow[rd,xshift=0.4ex,yshift=0.4ex] && \sf B \arrow[lu,xshift=-0.4ex] \arrow[ld,xshift=0.2ex,yshift=-0.2ex]  \\ 
& \sf C \arrow[lu,xshift=-0.2ex,yshift=-0.2ex] \arrow[uu,xshift=-0.425ex] \arrow[ru,xshift=-0.4ex,yshift=0.4ex] 
\end{tikzcd}
\qquad \qquad
G_x \colon \quad
\begin{tikzcd}
& \sf A \arrow[dd] \arrow[dr] \\
\sf 0 \arrow[ur] && \sf B \arrow[dl] \\ 
& \sf C \arrow[lu]
\end{tikzcd}
\]
For the isolated network $G^\triangleright$, 
there exists a complex-balanced equilibrium $x^\triangleright \in \R^3_>$ (implied by $\delta^\triangleright=0$)
which is detailed-balanced (implied by thermodynamics)
and hence formally balanced.
For any $x\in\R^3_>$,
the condition for formal balance is given by
$k_{\sf A\to B} \, k_{\sf B\to C} \, k_{\sf C\to A}=k_{\sf A\to C} \, k_{\sf C\to B} \, k_{\sf B\to A}$.
Hence, any state $x\in\R^3_>$ is formally balanced and, by~\eqref{eq:impl1}, (strongly) cycle-balanced.
That is, any mixed graph $G^\triangleright_x$, induced by $G^\triangleright$ and $x$,
does not contain a (weakly) directed cycle,
and the same holds when $G^\triangleright$ is seen as a subnetwork of $G$; see below.

For the full network $G$, 
there exists a complex-balanced equilibrium $x \in \R^3_>$ (implied by $\delta=0$).
Assume that $x$ is not detailed-balanced,
in particular,
that the mixed graph $G_x$, induced by $G$ and $x$,
does not have $\sf C \lrl 0 \lrl A$ as a subgraph.
By complex balance (for the complex $\sf 0$),
$G_x$ has $\sf C \to 0 \to A$ (or, alternatively, $\sf A \to 0 \to C$) as a subgraph; see the right diagram.
By Theorem~\ref{thm:rrn2}, $x$~is not cycle-balanced,
that is, there exists a directed cycle in $G_x$.
%By the argument above,
%the subgraph $G^\triangleright_x$ does not contain
%a (weakly) directed cycle (such as $\sf A \to B \to C \to A$, $\sf A \to C \to B \to A$, 
%$\sf A \to B \to C \lrl A$, and $\sf A \to C \lrl B \lrl A$).
By the argument above,
the subgraph $G^\triangleright_x$ is \textbf{not} a (weakly) directed cycle.
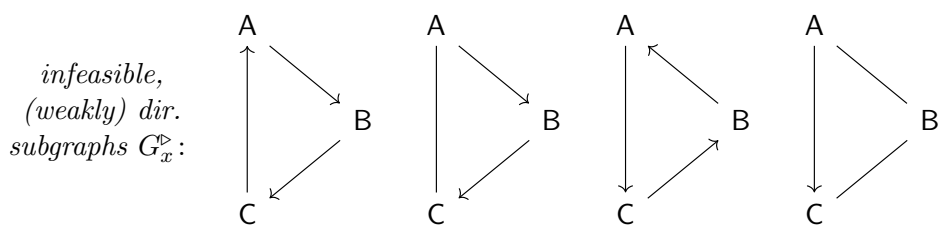
\[
\specialcell{infeasible, \\ (weakly) dir. \\ subgraphs $G^\triangleright_x \colon$}
\quad
\begin{tikzcd}
\sf A \arrow[dr] \\
& \sf B \arrow[dl] \\ 
\sf C \arrow[uu]
\end{tikzcd}
\quad
\begin{tikzcd}
\sf A \arrow[dr] \\
& \sf B \arrow[dl] \\ 
\sf C \arrow[dash,uu]
\end{tikzcd}
\quad
\begin{tikzcd}
\sf A \arrow[dd] \\
& \sf B \arrow[ul] \\ 
\sf C \arrow[ur]
\end{tikzcd}
\quad
\begin{tikzcd}
\sf A \arrow[dd] \\
& \sf B \arrow[dash,ul] \\ 
\sf C \arrow[dash,ur]
\end{tikzcd}
\]
The only feasible subgraph $G^\triangleright_x$ is $\sf C \leftarrow \sf A \to B \to C$;
see again the right diagram above.
The induced graph $G_x$ contains the directed cycles $\sf 0 \to A \to C \to 0$ and $\sf 0 \to A \to B \to C \to 0$
which involve the exchange reactions (in agreement with thermodynamics).

Remarkably, all edges of the induced graph (all directions of the net reactions) can be determined
without computing the complex-balanced equilibrium.
\end{exa}

%----------------------------------
\section{Balance in Markov chains}

The argument in Section \ref{sec:mix} has been developed for the application to reaction networks (RNs). 
However, owing to the abstractness of the result, it is easily applicable in any setting with an underlying graph structure. 
We illustrate this via Markov chains (MCs), a widely used class of stochastic models with a naturally associated graph. 
%\red{ In particular, we obtain a slightly stronger version of the well-known Kolmogorov cycle conditions for the existence of a reversible measure. }

%Say what is a Markov chain!

%The MC transitions from its present state to another state after a random time interval. The new state and the waiting time for the transition are determined by the transition probabilities (in the discrete time case) or by the transition rates (in the continuous time case). 

%In the following, we consider {\em continuous-time} MCs, % (on a {\em countable state space}),
%where transitions between states are determined by transition rates. %(for continuous time) or transition probabilities (for discrete time).
A conti\-nuous-time MC is a random process on a countable state space,
where a measure (in particular, a distribution) on the set of states %(and for nonnegative time) 
is determined by the initial measure and the transition rates %/probabilities. 
(via the Kolmogorov forward equations).
For a formal definition, see e.g.\ \cite{norris:markov}.
In a natural way, states can be viewed as vertices of a directed graph whose edges represent transitions with positive rates. %/probabilities. 
%From hereon, we only consider the continuous-time case (with transition rates), 
%however, every statement also holds for the discrete-time case
%(by replacing transition rate with transition probability and by obvious modifications). 

%{\bf Remark:} Unfortunately,
%the notion ``reversible'' has a different meaning in the RN and MC settings.
%Whereas a {\em reaction network} can be reversible,
%meaning that the underlying directed graph is symmetric,
%a {\em stationary measure} of a MC can be reversible (or detailed-balanced),
%meaning that, for every transition, the forward and backward rates are equal.
%(Necessarily, the associated directed graph is symmetric, that is, reversible in the RN sense.)
%Moreover, the notion ``complex-balanced'' does not exist in the MC setting
%since any stationary measure is complex-balanced in the RN sense,
%meaning that, for every state, the sums of incoming and outgoing transition rates are equal.

%Say what is the graph $G = (V,\mathcal T)$ and the transition probabilities!

We denote the set of states (vertices) by $V$ and the transition rate from state $x \in V$ to state $y \in V$ by $q(x,y)$.
Further, we introduce the set of transitions (edges) $\T$, that is, $(x , y) \in \T$ if $q(x,y) >0$. 
In the following, we require that $q(x,y) >0$ implies $q(y,x)>0$ for all $x,y \in V$. 
That is, we consider MCs where the associated simple, directed graph $G=(V,\T)$ is symmetric.
Such MCs are analogous to reversible RNs, however, we do not refer to them as ``reversible'' 
since this term is reserved for another notion; see below.

%Say what is a measure $\mu$!

%Say that stationary already implies ``complex-balanced''!

%Say what we mean by a ``reversible'' Markov chain: \\
%(unfortunately reversible already stands for detailed-balanced...)

%A {\em measure} $\mu$ is a function from subsets of $V$ to the non-negative reals and $\infty$, 
%$\mu \colon V \to \R_{\ge 0} \cup \{ \infty \}$.
A measure $\mu$ on the countable set $V$
assigns a nonnegative real or infinity to each subset of $V$.
Here, we consider only $\sigma$-finite measures where $\mu(\{x\}) < +\infty$ for all $x \in V$.
Following standard convention, we drop the curly brackets and write $\mu(x)$ for $\mu(\{x\})$.
If $\sum_{x \in V} \mu(x) = 1$, then $\mu$ is a distribution.
A measure $\mu$ is {\em stationary} if, for all $x \in V$, 
\[
\sum_{(x,y) \in \T} \mu(x) q(x,y) = \sum_{(y,x) \in \T} \mu(y) q(y,x) .
\]

A stationary measure of a MC 
%\red{(as a time-invariant solution of the Kolmogorov forward equations)} 
is analogous to a complex-balanced equilibrium of an RN 
in the sense that, for every state,
the sums of incoming and outgoing ``probability flows'' are equal.
Finally, a measure $\mu$ is {\em reversible} (detailed-balanced) if, for all $(x,y) \in \T$,
\[
\mu(x) q(x,y) =  \mu(y) q(y,x) .
\]
%If there exists a reversible measure \red{(distribution?)}, then the MC itself is called ``reversible''.
Clearly, the notions of detailed balance in RNs and MCs are analogous.

%Say what is the induced graph $G_\mu$!

Given a MC with associated symmetric, simple, directed graph $G=(V, \T)$
and a measure $\mu$,
the {\em induced graph} $G_\mu=(V,U,D)$ is a simple, mixed graph
defined as
\begin{align*}
(x \lrl y) \in U & \quad \text{if } (x,y) \in \T \text{ and } \mu(x) q(x,y) = \mu(y) q(y,x) ,\\
(x \to y) \in D & \quad \text{if } (x,y) \in \T \text{ and } \mu(x) q(x,y) > \mu(y) q(y,x) .
\end{align*} 
%See Section \ref{sec:mix}. 

%\red{explain edges in original graph and induced graph?}

\blue{
Now, let $\mu$ be a measure of a MC
and $G_\mu$ be the induced graph.
From the definitions in Section \ref{sec:mix},
we have the implications
\begin{equation*}
\begin{array}{lrl}
\mu \text{ is reversible} & \iff & G_\mu \text{ is edge-balanced,} \\[1ex]
\mu \text{ is stationary} & \implies & G_\mu \text{ is vertex-balanced.} %\\[1ex]
%\mu \text{ is sCycB} & \iff & G_\mu \text{ does not contain a weakly directed cycle,} \\[1ex]
%\mu \text{ is CycB} & \iff & G_\mu \text{ does not contain a directed cycle.}
\end{array}
\end{equation*}
}

An application of Theorem \ref{thm:main_inf} immediately yields the following result. 

\begin{thm} 
Let $G$ be the graph associated with a continuous-time Markov chain, % on a countable state space, 
where $q(x,y) >0$ if and only if $q(y,x)>0$. 
Let $\mu$ be a stationary measure.
If \blue{the induced graph} $G_\mu$ does not contain a directed cycle or a bi-infinite directed path,
then $\mu$ is a reversible measure. %detailed-balanced. 
\end{thm}

Its contrapositive is useful to state.
If a stationary measure is not reversible, then the induced graph contains a directed cycle or a bi-infinite directed path. 
See Examples \ref{exaMC1} and \ref{exaMC2} below.

%\red{
%\begin{cor} 
%Let $G$ be the graph associated with a continuous-time, positive recurrent Markov chain on a countable state space, 
%where $p(x,y) >0$ if and only if $p(y,x)>0$. 
%Let $\mu$ be a stationary measure.
%If $G_\mu$ does not contain a directed cycle, then $\mu$ is a reversible measure. %detailed-balanced. 
%\end{cor}
%\begin{proof}
%Let $\mu$ be a stationary measure which is not reversible on a positive-recurrent continuous-time Markov chain. 
%By the Kolmogorov cycle conditions \cite{durrett2010probability,kelly1979reversibility}, 
%there must be a cycle $(x_1, \ldots, x_n)$ such that $\prod_{i=1}^n p(x_i, x_{i+1}) > \prod_{i=1}^n p( x_{i+1}, x_i)$. 
%\end{proof}
%The corollary is a stronger version of the Kolmogorov cycle conditions for reversibility \cite{durrett2010probability,kelly1979reversibility}. 
%}

\begin{exa} \label{exaMC1}
Consider again the reversible cyclic network $\sf A \lras B \lras C \lras A$,
but this time with \textbf{stochastic} mass-action kinetics.
The corresponding rate constants are specified as edge labels in the graph below.
%\begin{equation*}\label{cycles}
%  \begin{tikzpicture}[baseline={(current bounding box.center)}]
%   %states
%   \node[state] (C)  at (0,0)  {\sf C};
%   \node[state] (A)  at (2,3)  {\sf A};
%   \node[state] (B)  at (4,0)  {\sf B};
%   %edges
%   \path[->]
%   (A) edge[bend left] node {$2$} (B)
%	edge[bend left] node {$1$} (C)
%	  (B) edge[bend left] node {$2$} (C)
%	edge[bend left] node {$1$} (A)
%	  (C) edge[bend left] node {$2$} (A)
%	edge[bend left] node {$1$} (B)
%;
%\end{tikzpicture}
%\end{equation*}
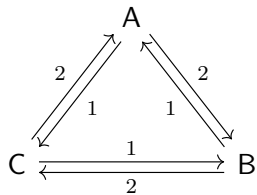
\[
\begin{tikzcd}
& \sf A \arrow[ddl,"1",xshift=0.3ex,yshift=-0.3ex] \arrow[ddr,"2",xshift=0.3ex,yshift=0.3ex] \\ \\
\sf C \arrow[rr,"1",yshift=0.425ex] \arrow[uur,"2",xshift=-0.3ex,yshift=0.3ex] && \sf B \arrow[uul,"1",xshift=-0.3ex,yshift=-0.3ex] \arrow[ll,"2",yshift=-0.425ex]
\end{tikzcd}
\]
The (infinite) graph $G=(V,\T)$ associated with the Markov chain is given by $V= \Z^3_\ge$,
$q\left((a,b,c)\to(a-1,b+1,c)\right)=2a$, $q\left((a,b,c)\to(a+1,b-1,c)\right)=b$, 
$q\left((a,b,c)\to(a-1,b,c+1)\right)=a$, etc.

For the deterministic system,
$x = (1,1,1)$ is a complex-balanced, but not detailed-balanced equilibrium. 
For the stochastic system,
the stationary (necessarily ``complex-balanced'') distribution $\pi \colon \Z^3_\ge \to \R$ is given by the product form 
\[
\pi(a,b,c) =  \frac{{\rm e}^{-3}}{a! \, b! \, c!} \, ,
\]
cf.~\cite{anderson2010product}.
Since this stationary distribution is not reversible (detailed-balanced), 
the induced graph $G_\pi$ must have a directed cycle or a bi-infinite directed path. 
Indeed, the (infinite) induced graph %(infinite) set of transitions 
can be decomposed into directed cycles (connecting three vertices),
as shown in the graph below. 
The corresponding net probability flows between states
are specified as edge labels.
%\[
%\begin{tikzcd}
%& (a+1,b,c) \arrow[dddr,""] \\ \\ \\
%(a,b,c+1) \arrow[uuur,""] && (a,b+1,c) \arrow[ll,""]
%\end{tikzcd}
%\]
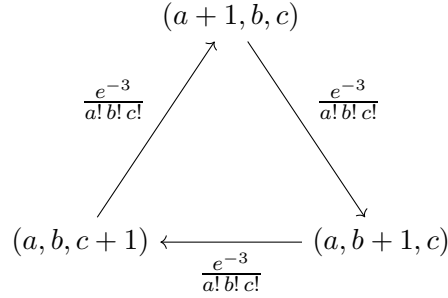
\begin{equation*}%\label{cycles}
  \begin{tikzpicture}[baseline={(current bounding box.center)}]
   %states
   \node[state] (C)  at (0,0)  {$(a,b,c+1)$};
   \node[state] (A)  at (2,3)  {$(a+1,b,c)$};
   \node[state] (B)  at (4,0)  {$(a,b+1,c)$};
   %edges
   \path[->]
        	 (A) edge node {$\frac{e^{-3} }{a!\,b!\,c!}$} (B)
	 (B) edge node {$\frac{e^{-3} }{a!\,b!\,c!}$} (C)
	 (C) edge node {$\frac{e^{-3} }{a!\,b!\,c!}$} (A)
;
\end{tikzpicture}
\end{equation*}
\end{exa}

%If the underlying graph of a Markov chain contains a bi-infinite directed path,
%then vertex balance and the absence of directed cycles in the induced graph do not guarantee edge balance. 
%We illustrate this with a simple example. 

\begin{exa}  \label{exaMC2}
Let $q \in (0,1)$.
Consider a Markov chain given by $V = \Z$, 
$q(x,x+1) = 2q^{-\abs{x}}$ and $q(x,x-1) = q^{-\abs{x}}$ for $x \in \Z$, and $q(x,x')=0$ otherwise.
Obviously, there are no directed cycles in the associated graph~$G$, except for the trivial two-cycles.
A~stationary distribution on $\Z$ is
\[
\pi(x) = \pi(0) q^{\abs{x}}
\]
with normalization constant $\pi(0)>0$. 
However, this distribution is not reversible (detailed-balanced), since
$\pi(x)q(x,x+1) \ne \pi(x+1)q(x+1,x)$ for any $x \in \Z$. 
Hence, the induced graph $G_\pi$ has directed edges $x \to x+1$ for $x \in \Z$.
The induced graph is vertex-balanced, but not edge-balanced,
in particular, $G_\pi$ contains a bi-infinite directed path.

Since $q(x,y) > 0$ if and only if $q(y,x)>0$ and there are no (nontrivial) cycles, % in the underlying graph of the Markov chain,
there must be a reversible stationary measure on $\Z$ as well. In fact,
\[
\rho(x) = \rho(0) 
\begin{cases} 
(2q)^x & \text{if } x \ge 0 \\
\left(\frac{q}{2}\right)^{-x} & \text{if } x < 0 
\end{cases}
\] 
is such a measure.
For $q<\frac{1}{2}$, it is finite and hence a distribution (for some normalization constant $\rho(0)>0$).
The induced graph $G_\rho$ is both vertex-balanced and edge-balanced.

Since there exist two different stationary distributions $\pi \ne \rho$, the Markov chain is not positive recurrent. 
\end{exa}

%\begin{example}
%Let $q \in (0,1)$.
%Consider a Markov chain given by $V = \Z$, 
%$q(x,x+1) = q \in (0,1)$ and $q(x,x-1) = 1-q$ for $x \in \Z$, and $q(x,x')=0$ otherwise.
%Obviously, there are no directed cycles in the associated graph $G$.
%A measure $\mu$ on $\Z$ is stationary if
%%\[
%%\sum_{x' \in \Z} \mu(x') \, q(x',x) = \mu(x) \sum_{x' \in \Z} q(x,x') ,
%%\]
%%that is,
%\[
%\mu(x) = \mu(x-1) \, q + \mu(x+1) \, (1-q) .
%\]
%For $q \neq \frac{1}{2}$, there are at least two distinct stationary measures,
%\begin{enumerate}
%\item $\mu \equiv 1$,
%\item $\mu(x) = \ds \left(\frac{q}{1-q} \right)^x$, $x \in \Z$. 
%\end{enumerate}
%(And the induced graphs $G_\mu$ are vertex-balanced.)
%However,
%only the second measure is detailed-balanced,
%\[
%\mu(x-1) q = \mu(x) (1-q) .
%\]
%(And the induced graph $G_\mu$ is edge-balanced.)
%\end{example}

Finally, we summarize similarities and dissimilarities in the settings of RNs and MCs %(and electrical networks) 
in a table.
\begin{table}[H]
\centering
\begin{adjustbox}{max width=\textwidth}
\begin{tabular}{|c||c|c|}
\hline

\specialcell{} & \specialcell{chemical \\ \bf reaction network \\ with mass-action \\ kinetics} & \specialcell{continuous-time \\ \bf Markov chain} \\ %& \specialcell{\bf electrical \\ \bf network}  \\
\hline
\hline

\specialcell{variable} & \specialcell{species \\ concentrations $x$} & \specialcell{probability \\ measure $\mu$} \\ %& potential $U$ \\
\hline
\hline

\specialcell{function \\ on vertex $i$} & monomial $x^{y(i)}$ & $\mu(i)$ \\ %& $U_i$ \\
\hline

\specialcell{function \\ on edge $ij$} & rate constant $k_{ij}$ & transition rate $q(i,j)$ \\ %& conductance $C_{ij}$ \\
\hline

\specialcell{product function \\ on edge $ij$} & \specialcell{reaction rate \\ $k_{ij} \, x^{y(i)}$} & \specialcell{probability flow \\ $\mu(i) \, q(i,j)$} \\ %& \specialcell{charge flow \\ $U_i \, C_{ij}$}  \\
\hline
\hline

\specialcell{vertex balance} & \specialcell{complex-balanced \\ equilibrium} & \specialcell{stationary \\ measure} \\ %& \specialcell{equilibrium \\ potential} \\
\hline

\specialcell{edge balance} & \specialcell{detailed-balanced \\ equilibrium} & \specialcell{reversible \\ measure} \\ 
\hline

\specialcell{cycle conditions}  & \specialcell{formal balance, \\ cycle balance} & \specialcell{Kolmogorov \\ cycle conditions} \\ %&  \specialcell{none \\ if $C_{ij} = C_{ji}$} \\
\hline
\end{tabular}
\end{adjustbox}
\end{table}

\subsection*{Acknowledgments}

SM was supported by the Austrian Science Fund (FWF), project % P28406 and 
P33218.
The paper benefited from discussions
at the workshop ``Advances in chemical reaction network theory''
at the Erwin Schr\"odinger Institute (ESI) in October 2018
and 
at the workshop (SQuaRE) ``Dynamical properties of deterministic and stochastic models of reaction networks''
at the American Institute of Mathematics (AIM) in March 2019.

%--------------
% Bibliography
%--------------

\bibliographystyle{abbrv}
\bibliography{crnt,BNbibliography}

%--------------
% End of document
%--------------

\end{document}